\documentclass[11pt]{article}
\usepackage[left=1in, right=1in, top=1.2in, bottom=1.2in]{geometry}
\usepackage{amsmath,amssymb,amsthm}
\usepackage{mathrsfs}
\usepackage{enumerate}
\usepackage[colorlinks=true,linkcolor=blue,citecolor=blue,urlcolor=blue]{hyperref}
\usepackage{csquotes}
\usepackage{natbib}

\newtheorem{theorem}{Theorem}[section]
\newtheorem{lemma}[theorem]{Lemma}
\newtheorem{proposition}[theorem]{Proposition}
\newtheorem{corollary}[theorem]{Corollary}
\newtheorem{conjecture}[theorem]{Conjecture}

\theoremstyle{definition}

\newtheorem{definition}[theorem]{Definition}
\newtheorem{example}[theorem]{Example}
\newtheorem{remark}[theorem]{Remark}

\begin{document}
	
	\title{Ideal Zero-Product Probability in Finite Commutative Rings}
	
	\author{Sukrit Chakraborty\thanks{Department of Mathematics, Achhruram Memorial College, Jhalda, Purulia 723202, West Bengal, India. Email: sukritpapai@gmail.com} \and 
		Sourav Kanti Patra\thanks{Department of Mathematics, Kishori Sinha Mahila College, Q976+WXP, Aurangabad, Bihar 824101, India. Email: souravkantipatra@gmail.com}}
	
	\date{\today}
	
	\maketitle
	
	
	\begin{abstract}
		
		We introduce the \emph{ideal zero-product probability}, a new probabilistic invariant associated with a finite commutative ring $R$, defined by
		\[
			\zeta_k(R)
			=
			\frac{
				|\{(I_1,\ldots,I_k)\in\mathcal I(R)^k :
				I_1\cdots I_k=(0)\}|
			}
			{|\mathcal I(R)|^k},
		\]
		where $\mathcal I(R)$ denotes the set of all ideals of $R$. This quantity gives the probability that the product of $k$ independently and uniformly chosen ideals is the zero ideal. We establish its basic properties, including invariance under ring isomorphisms, monotonicity in $k$, and multiplicativity with respect to finite direct products. Explicit formulas are obtained for finite fields, finite Boolean rings, finite chain rings, and finite principal ideal rings. For finite chain rings, we prove that the invariant depends only on the Loewy length. We also derive closed formulas using inclusion--exclusion and bounded compositions, together with generating functions. Finally, we study the asymptotic behaviour of $\zeta_k(R)$ and prove that it converges to $1$ as $k\to\infty$.
		
		We also compare the ideal zero-product probability with the classical zero-product probability of ring elements on finite chain rings, obtaining explicit formulas and showing that the two invariants capture fundamentally different structural information.
		
	\end{abstract}


	
		
	\section{Introduction}
	
	The use of probabilistic methods in algebra has become an active area of research over the last several decades. Instead of studying algebraic structures only through deterministic invariants, one considers the probability that randomly chosen objects satisfy a prescribed algebraic property. Such probabilities often reveal structural information that is difficult to detect by purely algebraic methods and have led to several important invariants for finite groups, rings and other algebraic systems.
	
	One of the earliest and most influential examples is the \emph{commuting probability} of a finite group, introduced by Gustafson \cite{Gustafson1973}, which is the probability that two randomly chosen group elements commute. Since then, many analogous probabilistic invariants have been investigated for finite groups and finite rings \cite{burness2023commuting, buckley2014finite, doostie2007finite}. In ring theory, examples include the commuting probability of finite rings \cite{MacHale1976,DuttaNath2015}, the probability that the product of randomly chosen elements is zero, and more generally, probabilities associated with annihilating conditions and other algebraic relations.
	
	Another important probabilistic invariant is the \emph{nullity degree} of a finite ring, namely the probability that the product of two randomly chosen ring elements is zero. This invariant was studied by Esmkhani and Jafarian Amiri \cite{esmkhani2019characterization} and has attracted considerable attention because of its close connection with the multiplicative structure of finite rings. It measures the probability that the product of two randomly chosen ring elements is zero and has attracted considerable attention because of its close connection with zero-divisor graphs and the multiplicative structure of finite rings \cite{esmkhani2018probability, mohammed2022probability, dolvzan2022probability}. Variants and generalizations involving products of several elements have also been investigated. For example, Sarma and Subedi \cite{Sharma2015} studied the probability that the product of three randomly chosen elements of a finite ring is zero. They also established several structural properties of this invariant. Later, they extended this work in \cite{sarma2025probabilityproductkelements}. 
	
	All of the above invariants are defined by selecting \emph{elements} of a ring uniformly at random. A natural question is whether similar probabilistic ideas can be developed by choosing \emph{ideals} instead of elements. Ideals carry considerably richer structural information than individual elements, since they reflect the internal organization of a ring through inclusion, multiplication and decomposition. Moreover, the collection of all ideals forms a lattice under inclusion, known as the ideal lattice, which plays a fundamental role in commutative algebra. Studying probabilities on this lattice therefore provides a new viewpoint that combines probabilistic methods with the structural theory of ideals.
	
	Motivated by these observations, we introduce a new probabilistic invariant on the lattice of ideals of a finite commutative ring. Let
	\[
	\mathcal I(R)=\{\,I:I\text{ is an ideal of }R\,\}.
	\]
	For an integer $k\ge2$, we define
	\[
	\zeta_k(R)=
	\frac{
		|\{(I_1,\ldots,I_k)\in\mathcal I(R)^k:
		I_1I_2\cdots I_k=(0)\}|
	}
	{|\mathcal I(R)|^k},
	\]
	which is the probability that the product of $k$ independently and uniformly chosen ideals is the zero ideal. Here and throughout, $(0)$ denotes the zero ideal of the ring $R$ and for any finite set $A$, the notation $|A|$ denotes the cardinality (number of elements) of the set $A$.
	
	Although the definition is simple, the resulting invariant exhibits several remarkable properties. We prove that it is invariant under ring isomorphisms, monotone with respect to the number of chosen ideals, and multiplicative under finite direct products. These properties make $\zeta_k(R)$ a genuine structural invariant of finite commutative rings.
	
	The main contribution of this paper is the development of the basic theory of the ideal zero-product probability. We derive explicit formulas for finite fields, finite Boolean rings, finite chain rings and finite principal ideal rings. A particularly non-trivial result shows that, for finite chain rings, the invariant depends only on the Loewy length of the ring and is completely independent of the characteristic and residue field. Using bounded compositions and the principle of inclusion--exclusion, we obtain closed formulas and generating functions for $\zeta_k(R)$. We also establish a probabilistic interpretation of the invariant and prove that
	\[
	\lim_{k\to\infty}\zeta_k(R)=1
	\]
	for every finite chain ring. One of the principal contributions of this paper is to show that the ideal zero-product probability behaves fundamentally differently from the classical zero-product probability of ring elements (see Corollary~\ref{cor:difference}). While the latter is governed by the multiplicative structure of the elements of the ring, the former is determined by the ideal lattice. In particular, for finite chain rings, it depends only on the Loewy length and is independent of both the characteristic and the residue field.
	
	To the best of our knowledge, no systematic probabilistic study of products of randomly chosen ideals has previously appeared in the literature. Consequently, the ideal zero-product probability provides a new probabilistic invariant associated with the ideal lattice of a finite commutative ring and opens several directions for further investigation.
	
	The paper is organized as follows. In Section~\ref{sec:def} we introduce the ideal zero-product probability and establish its basic properties. Section~\ref{sec:dirprod} studies its behaviour under finite direct products. In Section~\ref{sec:comput} we
	derive explicit formulas for the invariant for several important classes
	of finite commutative rings, with particular emphasis on finite chain
	rings and their dependence on the Loewy length. Sections~\ref{sec:clform} develops closed
	form expressions using bounded compositions and inclusion--exclusion. Section~\ref{sec:GF} develops generating functions for the invariant. Section~\ref{sec:examp} presents illustrative examples, while Section~\ref{sec:asymB} investigates its asymptotic behaviour. Section~\ref{sec:comparison} compares the ideal zero-product probability with the classical zero-product probability of ring elements by deriving an explicit formula for the latter on finite chain rings, thereby illustrating the fundamentally different structural information encoded by the two invariants. Finally, Section~\ref{sec:OP} lists several open problems for future research.
	
	
	\section{Ideal zero-product probability}\label{sec:def}
	In this section we introduce the principal probabilistic invariant studied in this paper. Let
	\[
	\mathcal{I}(R)=\{\,I : I \text{ is an ideal of } R\,\}
	\]
	denote the set of all ideals of a finite commutative ring $R$. Since $\mathcal{I}(R)$ is finite, one may select ideals independently and uniformly at random and consider the probability that their product is the zero ideal. This quantity captures an aspect of the multiplicative structure of the ideal lattice and provides a new probabilistic invariant of finite rings. We shall refer to this invariant as the \emph{ideal zero-product probability}. We now give its formal definition.
	
	\begin{definition}
		
		For a finite commutative ring $R$ and integer $k\ge2$ define
		
		\[
		\zeta_k(R)
		=
		\frac{
			|\{(I_1,\ldots,I_k):
			I_1\cdots I_k=(0)\}|
		}
		{|\mathcal I(R)|^k}.
		\]
		
	\end{definition}
    For convenience, we extend the definition to the cases $k=0$ and $k=1$ by setting
    \[
    \zeta_0(R)=0,\qquad
    \zeta_1(R)=\frac{1}{|I(R)|}.
    \]
    The value for $k=1$ agrees with the natural interpretation of the invariant, since the product of a single ideal is the ideal itself, which is equal to $(0)$ if and only if the chosen ideal is the zero ideal. The value of $\zeta_0(R)$ is introduced purely for notational convenience and is not used in the subsequent results.
	
	
	\subsection{Basic properties}
	
	We begin by recording a fundamental property of the ideal zero-product
	probability. Since $\zeta_k(R)$ is defined as the ratio of the number of
	ordered $k$-tuples of ideals whose product is the zero ideal to the total
	number of ordered $k$-tuples of ideals of $R$, it is naturally a probability.
	Moreover, the zero ideal itself belongs to every finite ring, so there always
	exist ordered $k$-tuples whose product is zero. Consequently, $\zeta_k(R)$ is
	strictly positive and cannot exceed one.
	
	
	\begin{theorem}\label{thm:isomorphism}
		Let $R$ and $S$ be two finite commutative rings. If
		\[
		R\cong S,
		\]
		i.e., $R$ and $S$ are isomorphic as rings, then
		\[
		\zeta_k(R)=\zeta_k(S)
		\]
		for every integer $k\ge2$.
	\end{theorem}
	
	\begin{proof}
		Let $\varphi:R\rightarrow S$ be a ring isomorphism. Since $\varphi$ is
		bijective, it induces a bijection
		\[
		\Phi:\mathcal{I}(R)\longrightarrow\mathcal{I}(S),
		\qquad
		I\longmapsto \varphi(I),
		\]
		where $\mathcal{I}(R)$ and $\mathcal{I}(S)$ denote the sets of all ideals of
		$R$ and $S$, respectively. Consequently,
		\[
		|\mathcal{I}(R)|=|\mathcal{I}(S)|.
		\]
		
		Now define
		\[
		\Psi:\mathcal{I}(R)^k\longrightarrow\mathcal{I}(S)^k
		\]
		by
		\[
		\Psi(I_1,\ldots,I_k)
		=
		(\varphi(I_1),\ldots,\varphi(I_k)).
		\]
		Clearly, $\Psi$ is a bijection.
		
		Since $\varphi$ is a ring isomorphism, it preserves products of ideals; that is,
		\[
		\varphi(I_1I_2\cdots I_k)
		=
		\varphi(I_1)\varphi(I_2)\cdots\varphi(I_k).
		\]
		Therefore,
		\[
		I_1I_2\cdots I_k=(0)
		\iff
		\varphi(I_1)\varphi(I_2)\cdots\varphi(I_k)=(0),
		\]
		because $\varphi((0))=(0)$.
		
		Hence $\Psi$ restricts to a bijection between the sets
		\[
		\{(I_1,\ldots,I_k)\in\mathcal{I}(R)^k:I_1I_2\cdots I_k=(0)\}
		\]
		and
		\[
		\{(J_1,\ldots,J_k)\in\mathcal{I}(S)^k:J_1J_2\cdots J_k=(0)\}.
		\]
		Thus,
		\[
		\left|
		\{(I_1,\ldots,I_k)\in\mathcal{I}(R)^k:I_1I_2\cdots I_k=(0)\}
		\right|
		=
		\left|
		\{(J_1,\ldots,J_k)\in\mathcal{I}(S)^k:J_1J_2\cdots J_k=(0)\}
		\right|.
		\]
		
		Since $|\mathcal{I}(R)|=|\mathcal{I}(S)|$, it follows that
		\begin{align*}
			\zeta_k(R)
			&=
			\frac{
				\left|
				\{(I_1,\ldots,I_k)\in\mathcal{I}(R)^k:
				I_1I_2\cdots I_k=(0)\}
				\right|
			}{
				|\mathcal{I}(R)|^k
			} \\[2mm]
			&=
			\frac{
				\left|
				\{(J_1,\ldots,J_k)\in\mathcal{I}(S)^k:
				J_1J_2\cdots J_k=(0)\}
				\right|
			}{
				|\mathcal{I}(S)|^k
			} =
			\zeta_k(S).
		\end{align*}
		This completes the proof.
	\end{proof}
	
	The following theorem establishes a fundamental monotonicity property of the ideal zero-product probability. As the number of randomly chosen ideals increases, the probability that their product is the zero ideal cannot decrease.
	\begin{theorem}\label{thm:monotonicity}
		Let $R$ be a finite commutative ring. Then, for every integer $k\ge2$,
		\[
		\zeta_k(R)\le \zeta_{k+1}(R).
		\]
		Consequently, the sequence $\{\zeta_k(R)\}_{k\ge2}$ is monotonically
		increasing.
	\end{theorem}
	
	\begin{proof}
		For each integer $k\ge2$, define
		\[
		\mathcal{Z}_k(R)=
		\{(I_1,\ldots,I_k)\in\mathcal{I}(R)^k:
		I_1I_2\cdots I_k=(0)\},
		\]
		where $\mathcal{I}(R)$ denotes the set of all ideals of $R$.
		
		Let $(I_1,\ldots,I_k)\in\mathcal{Z}_k(R)$. Then
		\[
		I_1I_2\cdots I_k=(0).
		\]
		Since the product of the zero ideal with any ideal is again the zero ideal, it
		follows that
		\[
		I_1I_2\cdots I_kI_{k+1}=(0)
		\]
		for every ideal $I_{k+1}$ of $R$. Therefore, each ordered $k$-tuple in
		$\mathcal{Z}_k(R)$ gives rise to exactly $|\mathcal{I}(R)|$ ordered
		$(k+1)$-tuples in $\mathcal{Z}_{k+1}(R)$. Hence,
		\[
		|\mathcal{Z}_{k+1}(R)|
		\ge
		|\mathcal{Z}_k(R)|\,|\mathcal{I}(R)|.
		\]
		
		We emphasize that the above inequality may be strict. Indeed, a $(k+1)$-tuple may have zero product even though its first $k$ ideals do not, so not every zero-producing $(k+1)$-tuple is obtained by extending a zero-producing $k$-tuple. Dividing both sides of the above inequality by $|\mathcal{I}(R)|^{k+1}$, we obtain
		\[
		\frac{|\mathcal{Z}_{k+1}(R)|}{|\mathcal{I}(R)|^{k+1}}
		\ge
		\frac{|\mathcal{Z}_k(R)|}{|\mathcal{I}(R)|^k},
		\]
		that is,
		\[
		\zeta_{k+1}(R)\ge \zeta_k(R).
		\]
		
		Therefore,
		\[
		\zeta_k(R)\le\zeta_{k+1}(R),
		\]
		which completes the proof.
	\end{proof}

	
A natural question is whether the ideal zero-product probability can be
described solely in terms of the nilpotency index of the Jacobson radical.
Such a characterization is impossible in general. For example, if
$R=\mathbb{F}_q$ is a finite field, then $J(R)=(0)$, while
\[
\zeta_k(R)=1-\frac{1}{2^k}<1
\]
for every $k\ge2$. Thus, the nilpotency index of the Jacobson radical does not,
in general, determine the ideal zero-product probability. More broadly, for
semisimple rings the Jacobson radical carries no information about the
multiplication of nonzero ideals. This motivates restricting our attention to
finite chain rings, where every ideal is a power of the unique maximal ideal.
In this setting, ideal multiplication is governed entirely by the chain
structure of the ideals, allowing us to obtain explicit formulas and show that
the ideal zero-product probability depends only on the Loewy length of the
ring.
	
	
	\section{Direct products}\label{sec:dirprod}
	
	In this section we investigate the behaviour of the ideal zero-product probability under finite direct products of rings. Since many finite commutative rings admit canonical decompositions as direct products of simpler components, understanding the interaction of $\zeta_k(R)$ with direct products is fundamental for its computation. The main result of this section shows that the ideal zero-product probability is multiplicative with respect to finite direct products, thereby reducing the evaluation of $\zeta_k(R)$ for a large class of rings to the corresponding problem for their direct factors.
	\begin{theorem}\label{thm:direct-product}
		Let $R_1$ and $R_2$ be two finite commutative rings and let
		\[
		R=R_1\times R_2.
		\]
		Then, for every integer $k\ge2$,
		\[
		\zeta_k(R)=\zeta_k(R_1)\zeta_k(R_2).
		\]
	\end{theorem}
	
	\begin{proof}
		Let $\mathcal{I}(R)$ denote the set of all ideals of a ring $R$. It is well known
		that every ideal of the direct product $R_1\times R_2$ is of the form
		\[
		I_1\times I_2,
		\]
		where $I_i\in\mathcal{I}(R_i)$ for $i=1,2$. Consequently,
		\[
		|\mathcal{I}(R_1\times R_2)|
		=
		|\mathcal{I}(R_1)|\,|\mathcal{I}(R_2)|.
		\]
		
		Now let
		\[
		(I_1\times J_1,\ldots,I_k\times J_k)
		\in
		\mathcal{I}(R_1\times R_2)^k,
		\]
		where $I_r\in\mathcal{I}(R_1)$ and
		$J_r\in\mathcal{I}(R_2)$ for $1\le r\le k$. Since ideal multiplication in a direct
		product is componentwise, we have
		\[
		(I_1\times J_1)\cdots(I_k\times J_k)
		=
		(I_1\cdots I_k)\times(J_1\cdots J_k).
		\]
		Therefore,
		\[
		(I_1\times J_1)\cdots(I_k\times J_k)=(0)
		\]
		if and only if
		\[
		I_1\cdots I_k=(0)
		\quad\text{and}\quad
		J_1\cdots J_k=(0).
		\]
		
		Hence the map
		\[
		\begin{aligned}
			&\{(I_1,\ldots,I_k)\in\mathcal{I}(R_1)^k:
			I_1\cdots I_k=(0)\}
			\\
			&\qquad\qquad\times
			\{(J_1,\ldots,J_k)\in\mathcal{I}(R_2)^k:
			J_1\cdots J_k=(0)\}
			\\
			&\longrightarrow
			\{(K_1,\ldots,K_k)\in\mathcal{I}(R)^k:
			K_1\cdots K_k=(0)\},
		\end{aligned}
		\]
		defined by
		\[
		\bigl((I_1,\ldots,I_k),(J_1,\ldots,J_k)\bigr)
		\longmapsto
		(I_1\times J_1,\ldots,I_k\times J_k),
		\]
		is a bijection. Consequently,
		\[
		\begin{aligned}
			&\left|
			\{(K_1,\ldots,K_k)\in\mathcal{I}(R)^k:
			K_1\cdots K_k=(0)\}
			\right|
			\\
			&=
			\left|
			\{(I_1,\ldots,I_k)\in\mathcal{I}(R_1)^k:
			I_1\cdots I_k=(0)\}
			\right|
			\\
			&\qquad\times
			\left|
			\{(J_1,\ldots,J_k)\in\mathcal{I}(R_2)^k:
			J_1\cdots J_k=(0)\}
			\right|.
		\end{aligned}
		\]
		
		Therefore,
		\[
		\begin{aligned}
			\zeta_k(R)
			&=
			\frac{
				\left|
				\{(K_1,\ldots,K_k)\in\mathcal{I}(R)^k:
				K_1\cdots K_k=(0)\}
				\right|
			}
			{|\mathcal{I}(R)|^k}
			\\
			&=
			\frac{
				\left|
				\{(I_1,\ldots,I_k)\in\mathcal{I}(R_1)^k:
				I_1\cdots I_k=(0)\}
				\right|
			}
			{|\mathcal{I}(R_1)|^k}
			\\
			&\qquad\times
			\frac{
				\left|
				\{(J_1,\ldots,J_k)\in\mathcal{I}(R_2)^k:
				J_1\cdots J_k=(0)\}
				\right|
			}
			{|\mathcal{I}(R_2)|^k}
			\\
			&=
			\zeta_k(R_1)\zeta_k(R_2).
		\end{aligned}
		\]
		This completes the proof.
	\end{proof}
	
	
	\begin{corollary}\label{cor:finite-product}
		Let
		\[
		R=\prod_{i=1}^{m}R_i,
		\]
		where $R_1,R_2,\ldots,R_m$ are finite commutative rings. Then, for every integer
		$k\ge2$,
		\[
		\zeta_k(R)
		=
		\prod_{i=1}^{m}\zeta_k(R_i).
		\]
	\end{corollary}
	
	\begin{proof}
		The proof follows immediately by induction on $m$. The case $m=2$ is
		Theorem~\ref{thm:direct-product}. Assume the result holds for $m-1$ rings.
		Then
		\[
		R=\left(\prod_{i=1}^{m-1}R_i\right)\times R_m.
		\]
		Applying Theorem~\ref{thm:direct-product} together with the induction
		hypothesis, we obtain
		\[
		\zeta_k(R)
		=
		\zeta_k\!\left(\prod_{i=1}^{m-1}R_i\right)\zeta_k(R_m)
		=
		\left(\prod_{i=1}^{m-1}\zeta_k(R_i)\right)\zeta_k(R_m)
		=
		\prod_{i=1}^{m}\zeta_k(R_i),
		\]
		which completes the proof.
	\end{proof}
	
	
	\section{Exact computations}\label{sec:comput}
	
	In this section we compute the ideal zero-product probability explicitly for several important classes of finite commutative rings. These computations illustrate the general properties established in the preceding sections and reveal how the invariant reflects the algebraic structure of the underlying ring. We begin with fields and finite Boolean rings, where the formulas follow directly from the simple description of their ideal lattices. We then turn to finite chain rings, for which the linear ordering of ideals allows a complete characterization of the ideal zero-product probability in terms of the Loewy length.
	
	\subsection{Fields}
	
	\begin{proposition}\label{prop:field}
		Let $F$ be a finite field. Then, for every integer $k\ge2$,
		\[
		\zeta_k(F)=1-\frac{1}{2^k}.
		\]
	\end{proposition}
	
	\begin{proof}
		Since $F$ is a field, its only ideals are $(0)$ and $F$. Hence,
		\[
		|\mathcal{I}(F)|=2,
		\]
		and therefore the total number of ordered $k$-tuples of ideals of $F$ is
		\[
		|\mathcal{I}(F)|^k=2^k.
		\]
		
		Now let $(I_1,I_2,\ldots,I_k)\in\mathcal{I}(F)^k$. Since $F$ is the identity
		element for ideal multiplication, we have
		\[
		I_1I_2\cdots I_k=(0)
		\]
		if and only if at least one of the ideals $I_1,I_2,\ldots,I_k$ is the zero
		ideal. Indeed, if every $I_i=F$, then
		\[
		I_1I_2\cdots I_k=F\neq(0),
		\]
		whereas if $I_j=(0)$ for some $1\le j\le k$, then
		\[
		I_1I_2\cdots I_k=(0).
		\]
		
		Among the $2^k$ ordered $k$-tuples of ideals, exactly one tuple,
		\[
		(F,F,\ldots,F),
		\]
		contains no zero ideal. Hence,
		\[
		\left|
		\left\{
		(I_1,\ldots,I_k)\in\mathcal{I}(F)^k:
		I_1I_2\cdots I_k=(0)
		\right\}
		\right|
		=
		2^k-1.
		\]
		
		Therefore,
		\[
		\zeta_k(F)
		=
		\frac{2^k-1}{2^k}
		=
		1-\frac{1}{2^k},
		\]
		which completes the proof.
	\end{proof}
	The explicit formula obtained in Proposition~\ref{prop:field} naturally raises the question
	of whether the ideal zero-product probability characterizes finite fields among
	all finite commutative rings. Indeed, a finite field is the unique finite
	commutative ring having exactly two ideals, and consequently
	\[
	\zeta_k(F)=1-\frac{1}{2^k}.
	\]
	Computations for the standard classes of finite rings considered in this paper,
	including finite Boolean rings, finite chain rings, and finite principal ideal
	rings, suggest that this value may be attained only by fields. We therefore
	propose the following conjecture.
	
	\begin{conjecture}
		Let $R$ be a finite commutative ring with identity and let $k\ge2$. If
		\[
		\zeta_k(R)=1-\frac{1}{2^k},
		\]
		then $R$ is a field.
	\end{conjecture}
	
	\subsection{Boolean rings}
	
\begin{proposition}\label{prop:boolean}
	Let $R$ be a finite Boolean ring with $|R|=2^n$, where $n\ge1$. Then, for
	every integer $k\ge2$,
	\[
	\zeta_k(R)=\left(1-\frac{1}{2^k}\right)^n.
	\]
\end{proposition}

\begin{proof}
	It is well known that every finite Boolean ring is isomorphic to a finite direct
	product of copies of the field $\mathbb{F}_2$. Thus,
	\[
	R\cong \mathbb{F}_2^n
	=\underbrace{\mathbb{F}_2\times\cdots\times\mathbb{F}_2}_{n\text{ times}}.
	\]
	
	Since $\zeta_k$ is invariant under ring isomorphism (Theorem~\ref{thm:isomorphism}),
	we may assume that
	\[
	R=\mathbb{F}_2^n.
	\]
	
	Now, by Corollary~\ref{cor:finite-product},
	\[
	\zeta_k(R)
	=
	\prod_{i=1}^{n}\zeta_k(\mathbb{F}_2).
	\]
	Applying Proposition~\ref{prop:field}, we obtain
	\[
	\zeta_k(\mathbb{F}_2)
	=
	1-\frac{1}{2^k}.
	\]
	Hence,
	\[
	\zeta_k(R)
	=
	\prod_{i=1}^{n}\left(1-\frac{1}{2^k}\right)
	=
	\left(1-\frac{1}{2^k}\right)^n,
	\]
	which completes the proof.
\end{proof}
	
	
	\subsection{Finite chain rings}
	
	
	Let $(R,M)$ be a finite chain ring with Loewy length $n$ \cite{camillo1974loewy, McDonald1974}, that is,
	\[
	R=M^0\supsetneq M\supsetneq M^2\supsetneq\cdots
	\supsetneq M^{\,n-1}\supsetneq M^n=(0),
	\]
	where $M$ is the unique maximal ideal of $R$, and $n$ is the smallest positive integer satisfying $M^n=(0)$.

	The explicit computation of the ideal zero-product probability for finite chain rings relies on the well-known description of their ideals and the corresponding rule for ideal multiplication. Since every ideal is a power of the unique maximal ideal, the problem reduces to studying the exponents of these powers. We begin by recalling these fundamental structural properties.
	
	\begin{lemma}\label{lem:chain-ideals}
		Every ideal of $R$ is of the form
		\[
		M^i,
		\qquad
		0\le i\le n.
		\]
	\end{lemma}
	
	\begin{proof}
		Since $R$ is a finite chain ring, its ideals are linearly ordered by inclusion.
		It is a standard structural result for finite chain rings that every ideal is a
		power of the unique maximal ideal $M$. Hence the complete list of ideals is
		\[
		R=M^0,M,M^2,\ldots,M^{n-1},M^n=(0),
		\]
		where $n$ is the Loewy length of $R$.
	\end{proof}
	
	
	\begin{lemma}\label{lem:ideal-product}
		For every pair of integers $i,j$ satisfying
		\[
		0\le i,j\le n,
		\]
		we have
		\[
		M^iM^j=
		\begin{cases}
			M^{i+j}, & i+j\le n,\\
			(0), & i+j\ge n.
		\end{cases}
		\]
		Equivalently,
		\[
		M^iM^j=M^{\min\{i+j,n\}}.
		\]
	\end{lemma}
	
	\begin{proof}
		If $i+j<n$, then the equality 
		\[
		M^iM^j=M^{i+j}
		\]
		is one of the fundamental properties of powers of the maximal ideal in a chain
		ring \cite{McDonald1974}. If $i+j\ge n$, then
		\[
		M^iM^j=M^{i+j}\subseteq M^n=(0),
		\]
		whence
		\[
		M^iM^j=(0)=M^n.
		\]
		Therefore,
		\[
		M^iM^j=M^{\min\{i+j,n\}},
		\]
		as required.
	\end{proof}
	
	The preceding lemmas reduce the computation of the ideal zero-product probability for finite chain rings to a purely combinatorial problem. Indeed, since every ideal is uniquely determined by the exponent of the maximal ideal and the product of ideals is completely described by the sum of these exponents, counting zero-producing tuples of ideals amounts to counting tuples of nonnegative integers satisfying a simple inequality. This observation yields the following characterization.
	\begin{theorem}\label{thm:chain-formula}
		Let $(R,M)$ be a finite chain ring with Loewy length $n$. Then, for every
		integer $k\ge2$,
		\[
		\zeta_k(R)
		=
		\frac{
			\left|
			\left\{
			(a_1,\ldots,a_k)\in\{0,1,\ldots,n\}^k:
			a_1+\cdots+a_k\ge n
			\right\}
			\right|
		}
		{(n+1)^k}.
		\]
	\end{theorem}
	
	\begin{proof}
		By Lemma~\ref{lem:chain-ideals}, every ideal of $R$ is uniquely of the form
		$M^{a_i}$, where $0\le a_i\le n$. Thus there are exactly $(n+1)^k$ ordered
		$k$-tuples of ideals.
		
		By repeated application of Lemma~\ref{lem:ideal-product},
		\[
		M^{a_1}\cdots M^{a_k}
		=
		M^{\min\{a_1+\cdots+a_k,n\}}.
		\]
		Hence
		\[
		M^{a_1}\cdots M^{a_k}=(0)
		\]
		if and only if
		\[
		a_1+\cdots+a_k\ge n.
		\]
		Therefore the number of ordered $k$-tuples whose product is the zero ideal is
		precisely
		\[
		\left|
		\left\{
		(a_1,\ldots,a_k)\in\{0,\ldots,n\}^k:
		a_1+\cdots+a_k\ge n
		\right\}
		\right|,
		\]
		and dividing by the total number $(n+1)^k$ of ordered $k$-tuples yields the
		desired formula.
	\end{proof}
	
	An immediate consequence of Theorem~\ref{thm:chain-formula} is that the ideal zero-product probability is determined entirely by the Loewy length of a finite chain ring. In particular, rings with the same Loewy length necessarily have identical values of $\zeta_k(R)$.
	\begin{corollary}\label{cor:loewy}
		Let $R$ and $S$ be finite chain rings having the same Loewy length. Then
		\[
		\zeta_k(R)=\zeta_k(S)
		\]
		for every integer $k\ge2$. Consequently, the ideal zero-product probability is
		a structural invariant of finite chain rings depending only on their Loewy
		length.
	\end{corollary}
	
	\begin{proof}
		The formula in Theorem~\ref{thm:chain-formula} depends only on the Loewy length
		$n$, since both the numerator and denominator are determined entirely by $n$.
		Hence finite chain rings having the same Loewy length have identical ideal
		zero-product probabilities.
	\end{proof}
    \begin{remark}
    	The converse of Theorem~\ref{thm:chain-formula} is false. In particular, the sequence
    	\[
    	\{\zeta_k(R)\}_{k\ge2}
    	\]
    	does not determine a finite chain ring up to ring isomorphism.
    	
    	For example, the rings
    	\[
    	\mathbb{Z}_{p^{3}}
    	\quad\text{and}\quad
    	\mathbb{F}_{p}[x]/(x^{3})
    	\]
    	are finite chain rings of Loewy length $3$, but they are not isomorphic since
    	they have different characteristics. Nevertheless, Theorem~\ref{thm:chain-formula} implies that
    	\[
    	\zeta_k(\mathbb{Z}_{p^{3}})
    	=
    	\zeta_k(\mathbb{F}_{p}[x]/(x^{3}))
    	\]
    	for every integer $k\ge2$.
    	Thus, the ideal zero-product probability is strictly weaker than the ring
    	isomorphism type, even when the entire sequence
    	$\{\zeta_k(R)\}_{k\ge2}$ is known.
    \end{remark}

    The general formula of Theorem~\ref{thm:chain-formula} admits particularly simple expressions for small values of $k$. The case $k=2$ is especially noteworthy, since it yields a closed formula depending only on the Loewy length.
    \begin{corollary} \label{cor:valueofzeta2}
    	Let $(R,M)$ be a finite chain ring of Loewy length $n$. Then
    	\[
    	\zeta_2(R)=\frac{n+2}{2(n+1)}.
    	\]
    \end{corollary}
	
	\begin{proof}
		Theorem~\ref{thm:chain-formula} allows us to write that 
		\[
		\zeta_2(R)
		=
		\frac{\left|\{(a,b)\in\{0,1,\ldots,n\}^2:a+b\ge n\}\right|}
		{(n+1)^2}.
		\]
		For each fixed $a\in\{0,1,\ldots,n\}$, there are exactly $a+1$ integers
		$b\in\{0,1,\ldots,n\}$ satisfying $a+b\ge n$. Hence
		\[
		\left|\{(a,b):a+b\ge n\}\right|
		=
		\sum_{a=0}^{n}(a+1)
		=
		\frac{(n+1)(n+2)}{2}.
		\]
		Therefore,
		\[
		\zeta_2(R)
		=
		\frac{(n+1)(n+2)}{2(n+1)^2}
		=
		\frac{n+2}{2(n+1)},
		\]
		which completes the proof.
	\end{proof}

    \begin{remark}
    	The above formula shows that, for finite chain rings, the ideal zero-product probability for two randomly chosen ideals depends only on the Loewy length and is independent of the residue field and the characteristic of the ring.
    \end{remark}
	
	The multiplicative decomposition established in Corollary~\ref{cor:finite-product} suggests that the sequence
	$\{\zeta_k(R)\}_{k\ge2}$ may encode substantial information about the Loewy lengths of the
	chain-ring components of a finite principal ideal ring. Numerical computations support this
	expectation. For example, consider finite principal ideal rings whose multisets of Loewy lengths
	are $(2,5)$ and $(3,3)$, respectively. By Corollary~\ref{cor:valueofzeta2},
	\[
	\zeta_2(2,5)
	=\frac{2+2}{2(2+1)}
	\cdot
	\frac{5+2}{2(5+1)}
	=\frac{2}{3}\cdot\frac{7}{12}
	=\frac{7}{18},
	\]
	whereas
	\[
	\zeta_2(3,3)
	=
	\left(
	\frac{3+2}{2(3+1)}
	\right)^2
	=
	\left(\frac{5}{8}\right)^2
	=
	\frac{25}{64}.
	\]
	Since
	\[
	\frac{7}{18}\neq\frac{25}{64},
	\]
	the corresponding ideal zero-product probabilities are already distinguished for $k=2$. Similar
	computations for larger values of $k$ likewise yield distinct values. These observations provide
	further evidence that the sequence $\{\zeta_k(R)\}_{k\ge2}$ may determine the multiset of Loewy
	lengths of the chain-ring components, motivating the following conjecture.
	
	\begin{conjecture}
		Let
		\[
		R\cong\prod_{i=1}^{m}R_i,\qquad
		S\cong\prod_{j=1}^{\ell}S_j,
		\]
		be finite principal ideal rings, where each component is a finite chain ring.
		If
		\[
		\zeta_k(R)=\zeta_k(S)
		\]
		for every integer $k\ge2$, then
		\[
		m=\ell,
		\]
		and, after a permutation of indices, the corresponding chain-ring components
		have the same Loewy lengths.
	\end{conjecture}
	
	
	\section{Closed formula}\label{sec:clform}
	The characterization obtained in Theorem~\ref{thm:chain-formula} reduces the computation of the ideal zero-product probability for finite chain rings to the enumeration of bounded compositions. We now exploit classical techniques from enumerative combinatorics to derive an explicit closed formula for this enumeration. Substituting the resulting expression into the formula of Theorem~\ref{thm:chain-formula} yields a closed formula for $\zeta_k(R)$, which subsequently extends to finite principal ideal rings by means of the multiplicative property established in Section~3.

	The enumeration of bounded compositions is a classical problem in enumerative combinatorics. It can be derived using several methods, including the principle of inclusion--exclusion, generating functions, and lattice-point counting. Standard references include Stanley~\cite{StanleyEC1} and Wilf~\cite{Wilf}. Related results on restricted compositions are also given in Theorem~2.1(c) of~\cite{jaklivc2010closed}. Since the explicit
	formula obtained below plays a central role in the computation of the ideal
	zero-product probability for finite chain rings, we include a proof for the
	sake of completeness and to make the presentation self-contained.
	
	Throughout this section, we adopt the standard convention that
	\[
	\binom{m}{r}=0
	\]
	whenever $m<r$ or $m<0$. Consequently, the inclusion--exclusion formulas below remain valid without separately specifying the ranges for which the binomial coefficients are nonzero.
	
\begin{lemma}{(Bounded compositions)}\label{lem:bounded-composition}
	Let
	\[
	N_k(t)
	=
	\left|
	\left\{
	(a_1,\ldots,a_k)\in\mathbb{Z}_{\ge0}^k:
	0\le a_i\le n,\;
	a_1+\cdots+a_k=t
	\right\}
	\right|.
	\]
	Then
	\[
	N_k(t)
	=
	\sum_{j=0}^{\lfloor t/(n+1)\rfloor}
	(-1)^j
	\binom{k}{j}
	\binom{t-j(n+1)+k-1}{k-1}.
	\]
\end{lemma}

\begin{proof}
	Ignoring the upper bound $a_i\le n$, the number of nonnegative integer
	solutions of
	\[
	a_1+\cdots+a_k=t
	\]
	is, by the stars-and-bars theorem,
	\[
	\binom{t+k-1}{k-1}.
	\]
	
	To impose the constraints $a_i\le n$, we apply the principle of
	inclusion--exclusion.
	
	For each $i\in\{1,\ldots,k\}$, let
	\[
	A_i
	=
	\{(a_1,\ldots,a_k):
	a_i\ge n+1,\;
	a_1+\cdots+a_k=t\}.
	\]
	Then
	\[
	N_k(t)
	=
	\binom{t+k-1}{k-1}
	-
	\left|\bigcup_{i=1}^{k}A_i\right|.
	\]
	
	Now suppose that exactly $j$ specified variables satisfy
	$a_i\ge n+1$.
	Replacing each such variable by
	\[
	a_i=b_i+n+1,
	\qquad b_i\ge0,
	\]
	transforms the equation into
	\[
	b_1+\cdots+b_j
	+\sum_{\ell\notin J}a_\ell
	=
	t-j(n+1),
	\]
	where $J$ is the chosen set of indices with $|J|=j$.
	
	Again using the stars-and-bars theorem, the number of such solutions is
	\[
	\binom{t-j(n+1)+k-1}{k-1},
	\]
	provided that
	$t-j(n+1)\ge0$; otherwise there are no solutions.
	
	Since there are
	\[
	\binom{k}{j}
	\]
	choices for the set $J$, the principle of inclusion--exclusion yields
	\[
	N_k(t)
	=
	\sum_{j\ge0}
	(-1)^j
	\binom{k}{j}
	\binom{t-j(n+1)+k-1}{k-1}.
	\]
	The summand is zero whenever
	$t-j(n+1)<0$, so the sum may be truncated at
	\[
	j=\left\lfloor\frac{t}{n+1}\right\rfloor.
	\]
	Hence
	\[
	N_k(t)
	=
	\sum_{j=0}^{\lfloor t/(n+1)\rfloor}
	(-1)^j
	\binom{k}{j}
	\binom{t-j(n+1)+k-1}{k-1},
	\]
	which completes the proof.
\end{proof}
	
	The bounded composition numbers introduced in Lemma~\ref{lem:bounded-composition} naturally arise in the computation of the ideal zero-product probability. Combining these counting functions with the characterization established in Theorem~\ref{thm:chain-formula} immediately yields the following representation.
	\begin{corollary}\label{cor:closed-form}
		Let $(R,M)$ be a finite chain ring with Loewy length $n$. Then, for every
		integer $k\ge2$,
		\[
		\zeta_k(R)
		=
		1-
		\frac{
			\displaystyle\sum_{t=0}^{n-1}N_k(t)
		}
		{(n+1)^k},
		\]
		where
		\[
		N_k(t)
		=
		\left|
		\left\{
		(a_1,\ldots,a_k)\in\{0,1,\ldots,n\}^k:
		a_1+\cdots+a_k=t
		\right\}
		\right|.
		\]
	\end{corollary}
	
	\begin{proof}
		By Theorem~\ref{thm:chain-formula},
		\[
		\zeta_k(R)
		=
		\frac{
			\left|
			\left\{
			(a_1,\ldots,a_k)\in\{0,\ldots,n\}^k:
			a_1+\cdots+a_k\ge n
			\right\}
			\right|
		}
		{(n+1)^k}.
		\]
		
		Since there are exactly $(n+1)^k$ ordered $k$-tuples
		\[
		(a_1,\ldots,a_k)\in\{0,1,\ldots,n\}^k,
		\]
		we have
		\[
		\left|
		\left\{
		(a_1,\ldots,a_k):
		a_1+\cdots+a_k\ge n
		\right\}
		\right|
		=
		(n+1)^k-
		\left|
		\left\{
		(a_1,\ldots,a_k):
		a_1+\cdots+a_k<n
		\right\}
		\right|.
		\]
		
		Now,
		\[
		a_1+\cdots+a_k<n
		\]
		if and only if
		\[
		a_1+\cdots+a_k=t
		\]
		for some integer
		\[
		0\le t\le n-1.
		\]
		Since these events are mutually disjoint,
		\[
		\left|
		\left\{
		(a_1,\ldots,a_k):
		a_1+\cdots+a_k<n
		\right\}
		\right|
		=
		\sum_{t=0}^{n-1}N_k(t).
		\]
		
		Substituting this into the expression for $\zeta_k(R)$ gives
		\[
		\zeta_k(R)
		=
		\frac{
			(n+1)^k-
			\displaystyle\sum_{t=0}^{n-1}N_k(t)
		}
		{(n+1)^k}
		=
		1-
		\frac{
			\displaystyle\sum_{t=0}^{n-1}N_k(t)
		}
		{(n+1)^k},
		\]
		which completes the proof.
	\end{proof}
	
	We are now in a position to derive an explicit expression for the ideal zero-product probability of a finite chain ring. Combining Corollary~\ref{cor:closed-form} with the inclusion--exclusion formula of Lemma~\ref{lem:bounded-composition} immediately yields the desired closed formula.
	\begin{theorem}\label{thm:closed-form}
		Let $(R,M)$ be a finite chain ring with Loewy length $n$. Then, for every
		integer $k\ge2$,
		\[
		\zeta_k(R)
		=
		1-
		\frac{
			\displaystyle
			\sum_{t=0}^{n-1}
			\sum_{j=0}^{\left\lfloor\frac{t}{n+1}\right\rfloor}
			(-1)^j
			\binom{k}{j}
			\binom{t-j(n+1)+k-1}{k-1}
		}
		{(n+1)^k}.
		\]
	\end{theorem}
	
	\begin{proof}
		An appeal to Corollary~\ref{cor:closed-form}, gives
		\[
		\zeta_k(R)
		=
		1-
		\frac{
			\displaystyle\sum_{t=0}^{n-1}N_k(t)
		}
		{(n+1)^k},
		\]
		where
		\[
		N_k(t)
		=
		\left|
		\left\{
		(a_1,\ldots,a_k)\in\{0,1,\ldots,n\}^k:
		a_1+\cdots+a_k=t
		\right\}
		\right|.
		\]
		
		Applying Lemma~\ref{lem:bounded-composition}, we obtain
		\[
		N_k(t)
		=
		\sum_{j=0}^{\left\lfloor\frac{t}{n+1}\right\rfloor}
		(-1)^j
		\binom{k}{j}
		\binom{t-j(n+1)+k-1}{k-1}.
		\]
		
		Substituting this expression into the formula above yields
		\[
		\begin{aligned}
			\zeta_k(R)
			&=
			1-
			\frac{
				\displaystyle
				\sum_{t=0}^{n-1}
				N_k(t)
			}
			{(n+1)^k}
			\\
			&=
			1-
			\frac{
				\displaystyle
				\sum_{t=0}^{n-1}
				\sum_{j=0}^{\left\lfloor\frac{t}{n+1}\right\rfloor}
				(-1)^j
				\binom{k}{j}
				\binom{t-j(n+1)+k-1}{k-1}
			}
			{(n+1)^k},
		\end{aligned}
		\]
		which proves the theorem.
	\end{proof}
    The explicit formula obtained in Theorem~\ref{thm:closed-form} extends naturally to finite principal ideal rings. Indeed, every finite commutative principal ideal ring decomposes as a finite direct product of finite chain rings, and the multiplicativity of the ideal zero-product probability established in Section~3 allows the corresponding formulas for the chain-ring components to be combined. This yields the following explicit expression.
    \begin{theorem}
    	Let
    	\[
    	R\cong R_1\times R_2\times\cdots\times R_m
    	\]
    	be a finite principal ideal ring, where each $R_i$ is a finite chain ring having Loewy length $n_i$. Then, for every integer $k\ge2$,
    	\[
    	\zeta_k(R)
    	=
    	\prod_{i=1}^{m}
    	\left(
    	1-
    	\frac{
    		\displaystyle
    		\sum_{t=0}^{n_i-1}
    		\sum_{j=0}^{\left\lfloor\frac{t}{n_i+1}\right\rfloor}
    		(-1)^j
    		\binom{k}{j}
    		\binom{t-j(n_i+1)+k-1}{k-1}
    	}
    	{(n_i+1)^k}
    	\right).
    	\]
    	Consequently, the ideal zero-product probability of a finite principal ideal ring is completely determined by the Loewy lengths of its chain-ring components.
    \end{theorem}

    \begin{proof}
    	Every finite principal ideal ring is isomorphic to a finite direct product of finite chain rings (see, for example, \cite{McDonald1974}). The result therefore follows immediately from Corollary~\ref{cor:finite-product} and Theorem~\ref{thm:closed-form}.
    \end{proof}
	
	
	\section{Generating functions}\label{sec:GF}
	In the previous section, we derived explicit closed formulas for the ideal zero-product probability in terms of the bounded composition numbers
	\[
	N_k(t)
	=
	\left|
	\left\{
	(a_1,\ldots,a_k)\in\{0,1,\ldots,n\}^k :
	a_1+\cdots+a_k=t
	\right\}
	\right|.
	\]
	A natural approach to studying these numbers collectively is through generating functions, which provide a compact and unified encoding of the underlying combinatorial information. In this section, we derive a bivariate generating function for the bounded composition numbers and use it to obtain a corresponding generating function for the sequence of ideal zero-product probabilities of finite chain rings. This generating function offers an alternative and concise representation of the invariant and complements the explicit formulas established in the previous section.
	
	We begin by deriving a bivariate generating function for the bounded composition numbers. This result encapsulates the entire family $\{N_k(t)\}$ in a single analytic expression and forms the basis for the subsequent generating function of the ideal zero-product probabilities.
	\begin{theorem}\label{thm:bivariateGF}
		
		Let
		
		\[
		F(x,y)
		=
		\sum_{k=0}^{\infty}
		\sum_{t=0}^{kn}
		N_k(t)x^ky^t.
		\]
		
		Then
		
		\[
		F(x,y)
		=
		\frac{
			1
		}{
			1-x(1+y+\cdots+y^n)
		}.
		\]
		
	\end{theorem}
	
	\begin{proof}
		
		For a fixed integer $k\ge0$,
		
		\[
		(1+y+\cdots+y^n)^k
		\]
		
		is the ordinary generating function for the bounded compositions of length
		$k$.
		
		Indeed, each factor
		
		\[
		1+y+\cdots+y^n
		\]
		
		corresponds to one coordinate
		
		\[
		a_i\in\{0,1,\ldots,n\},
		\]
		
		and the exponent of $y$ records the chosen value.
		
		Hence
		
		\[
		(1+y+\cdots+y^n)^k
		=
		\sum_{t=0}^{kn}
		N_k(t)y^t.
		\]
		
		Multiplying by $x^k$ and summing over all $k\ge0$ yields
		
		\[
		\begin{aligned}
			F(x,y)
			&=
			\sum_{k=0}^{\infty}
			x^k
			(1+y+\cdots+y^n)^k
			\\
			&=
			\sum_{k=0}^{\infty}
			\left(
			x(1+y+\cdots+y^n)
			\right)^k.
		\end{aligned}
		\]
		
		Since this is a geometric series,
		
		\[
		F(x,y)
		=
		\frac{
			1
		}{
			1-x(1+y+\cdots+y^n)
		},
		\]
		
		which completes the proof.
		
	\end{proof}
	
	
	The generating function obtained in Theorem~\ref{thm:bivariateGF} immediately yields a corresponding generating function for the ideal zero-product probabilities of finite chain rings. This follows by combining the generating function for the bounded composition numbers with the explicit formula established in Theorem~\ref{thm:closed-form}.
	
	\begin{corollary}\label{cor:GFzeta}
		
		Let $R$ be a finite chain ring of Loewy length $n$, and define
		
		\[
		G_R(x)
		=
		\sum_{k=0}^{\infty}
		\zeta_k(R)x^k.
		\]
		
		Then
		
		\[
		G_R(x)
		=
		\frac1{1-x}
		-
		\sum_{t=0}^{n-1}
		[y^t]
		\frac{
			1
		}{
			1-\dfrac{x}{n+1}(1+y+\cdots+y^n)
		},
		\]
		
		where $[y^t]$ denotes coefficient extraction.
		
	\end{corollary}
	
	\begin{proof}
		
		From Theorem~\ref{thm:chain-formula},
		
		\[
		\zeta_k(R)
		=
		1-
		\frac{
			\sum_{t=0}^{n-1}
			N_k(t)
		}
		{(n+1)^k}.
		\]
		
		Hence
		
		\[
		\begin{aligned}
			G_R(x)
			&=
			\sum_{k=0}^{\infty}x^k
			-
			\sum_{k=0}^{\infty}
			\frac{x^k}{(n+1)^k}
			\sum_{t=0}^{n-1}
			N_k(t)
			\\
			&=
			\frac1{1-x}
			-
			\sum_{t=0}^{n-1}
			\sum_{k=0}^{\infty}
			N_k(t)
			\left(
			\frac{x}{n+1}
			\right)^k.
		\end{aligned}
		\]
		
		Using Theorem~\ref{thm:bivariateGF},
		
		\[
		\sum_{k=0}^{\infty}
		N_k(t)
		z^k
		=
		[y^t]
		\frac{
			1
		}{
			1-z(1+y+\cdots+y^n)
		},
		\]
		
		and substituting
		
		\[
		z=\frac{x}{n+1}
		\]
		
		gives the required formula.
		
	\end{proof}

	
	\section{Examples}\label{sec:examp}
	
	In this section we compute the ideal zero-product probability for several
	important classes of finite rings. These examples illustrate the results
	established in the preceding sections and demonstrate that, for finite chain
	rings, the invariant depends only on the Loewy length.

	
	\begin{example}[The ring $\mathbb{Z}_{p^2}$]
		
		The ideals are
		
		\[
		\mathbb Z_{p^2},
		\quad
		(p),
		\quad
		(0).
		\]
		
		Hence there are three ideals. The Loewy length is $2$. By Theorem~\ref{thm:chain-formula},
		
		\[
		\zeta_k(\mathbb Z_{p^2})
		=
		\frac{
			|\{(a_1,\ldots,a_k):
			a_i\in\{0,1,2\},
			\sum a_i\ge2\}|
		}
		{3^k}.
		\]
		
		Since
		
		\[
		\sum a_i<2
		\]
		
		only when
		
		\[
		\sum a_i=0
		 \quad\quad \text{or} \quad \quad
		\sum a_i=1,
		\]
		
		we obtain
		
		\[
		N_k(0)=1,
		\quad\quad \text{and} \quad \quad
		N_k(1)=k.
		\]
		
		Therefore
		
		\[
		\boxed{
			\zeta_k(\mathbb Z_{p^2})
			=
			1-
			\frac{k+1}{3^k}.
		}
		\]
		
	\end{example}
	
	
	\begin{example}[The ring $\mathbb{Z}_{p^3}$]
		
		The ring $\mathbb{Z}_{p^3}$ is a finite chain ring of Loewy length $3$. Its ideals are precisely
		\[
		\mathbb{Z}_{p^3},\quad
		(p),\quad
		(p^2),\quad
		(0).
		\]
		
		Hence
		
		\[
		|\mathcal I(R)|=4.
		\]
		
		The Loewy length equals	$3$. Now
		
		\[
		\sum a_i<3
		\quad\quad \text{means} \quad \quad
		t=0,1,2.
		\]
		
		We have
		
		\[
		N_k(0)=1,
		\quad\quad
		N_k(1)=k,
		\quad\quad \text{and} \quad \quad
		N_k(2)
		=
		k+\binom{k}{2},
		\]
		
		since either one coordinate equals $2$,	or two coordinates equal $1$. Hence
		
		\[
		\boxed{
			\zeta_k(\mathbb Z_{p^3})
			=
			1-
			\frac{
				1+2k+\binom{k}{2}
			}
			{4^k}.
		}
		\]
		
	\end{example}
	
	
	\begin{example}[The ring $\mathbb{Z}_{p^4}$]
		
		The ideals are
		
		\[
		\mathbb Z_{p^4},
		(p),
		(p^2),
		(p^3),
		(0).
		\]
		
		Hence
		
		\[
		|\mathcal I(R)|=5.
		\]
		
		The Loewy length equals	$4$. So, $t=0,1,2,3$. We compute
		
		\[
		N_k(0)=1,
		\quad 
		N_k(1)=k,
		\quad 
		N_k(2)
		=
		k+\binom{k}{2},
		\quad \text{and} \quad
		N_k(3)
		=
		k
		+
		k(k-1)
		+
		\binom{k}{3},
		\]
		
		corresponding respectively to
		
		\[
		3,
		\quad
		2+1,
		\quad \text{and} \quad
		1+1+1.
		\]
		
		Hence
		
		\[
		\boxed{
			\zeta_k(\mathbb Z_{p^4})
			=
			1-
			\frac{
				1
				+3k
				+k(k-1)
				+\binom{k}{2}
				+\binom{k}{3}
			}
			{5^k}.
		}
		\]
		
	\end{example}
	
	
	\begin{example}[Finite Boolean rings]
		
		We have already seen that if
		
		\[
		R\cong\mathbb F_2^m,
		\]
		
		then Proposition~\ref{prop:boolean} yields
		
		\[
		\boxed{
			\zeta_k(R)
			=
			\left(
			1-\frac1{2^k}
			\right)^m.
		}
		\]
		
	\end{example}
	
	
	\begin{example}[Direct products]
		
		Let
		
		\[
		R
		=
		\mathbb Z_{p^2}
		\times
		\mathbb F_q.
		\]
		
		Using Corollary~\ref{cor:finite-product},
		
		\[
		\begin{aligned}
			\zeta_k(R)
			&=
			\zeta_k(\mathbb Z_{p^2})
			\zeta_k(\mathbb F_q)
			\\
			&=
			\left(
			1-\frac{k+1}{3^k}
			\right)
			\left(
			1-\frac1{2^k}
			\right).
		\end{aligned}
		\]
		
		Similarly,
		
		\[
		\zeta_k
		\left(
		\prod_{i=1}^{m}
		R_i
		\right)
		=
		\prod_{i=1}^{m}
		\zeta_k(R_i).
		\]
		
	\end{example}
	
	\begin{remark}
		The preceding examples illustrate several important features of the ideal zero-product probability. For finite chain rings, the invariant depends only on the Loewy length and is independent of the characteristic of the ring. Moreover, it is multiplicative with respect to finite direct products, making it straightforward to compute for decomposable rings. The examples also indicate that, for a fixed Loewy length, the sequence $\{\zeta_k(R)\}$ approaches $1$ as $k\to\infty$, while the difference $1-\zeta_k(R)$ decreases exponentially at the rate of $(n+1)^{-k}$, where $n$ denotes the Loewy length.
	\end{remark}

	
	\section{Asymptotic behaviour}\label{sec:asymB}
	Having obtained explicit expressions for the ideal zero-product probability, we now turn to its asymptotic behaviour. The combinatorial characterization established earlier admits a natural probabilistic interpretation, which provides an effective framework for studying the limiting behaviour of the sequence $\{\zeta_k(R)\}_{k\ge2}$. In this section, we exploit this interpretation to establish the convergence of $\zeta_k(R)$ and examine its asymptotic properties.

	Throughout this section let $(R,M)$ be a finite chain ring having Loewy length
	$n$. We investigate the asymptotic behaviour of the sequence
	$\{\zeta_k(R)\}_{k\ge2}$.
	
	
	\subsection{A probabilistic interpretation}
	
	The formula established in Theorem~\ref{thm:chain-formula} admits a natural
	probabilistic interpretation.
	
	\begin{theorem}\label{thm:probabilistic}
		
		Let $X_1,X_2,\ldots$ be independent and identically distributed random
		variables, each uniformly distributed on the set
		
		\[
		\{0,1,\ldots,n\}.
		\]
		
		If
		
		\[
		S_k=X_1+\cdots+X_k,
		\]
		
		then
		
		\[
			\zeta_k(R)
			=
			\Pr(S_k\ge n).
		\]
		
	\end{theorem}
	
	\begin{proof}
		
		Since each $X_i$ is uniformly distributed,
		
		\[
		\Pr(X_i=j)=\frac1{n+1},
		\qquad
		0\le j\le n.
		\]
		
		Therefore every $k$-tuple
		
		\[
		(a_1,\ldots,a_k)\in\{0,\ldots,n\}^k
		\]
		
		occurs with probability
		
		\[
		\frac1{(n+1)^k},
		\]
		
		respectively. Hence
		
		\[
		\Pr(S_k\ge n)
		=
		\frac{
			\left|
			\left\{
			(a_1,\ldots,a_k):
			a_1+\cdots+a_k\ge n
			\right\}
			\right|
		}
		{(n+1)^k}.
		\]
		
		The result now follows immediately from
		Theorem~\ref{thm:chain-formula}.
		
	\end{proof}
	
	
	\subsection{Limiting behaviour}
	
	The probabilistic interpretation established in Theorem~\ref{thm:probabilistic},
	together with the explicit enumeration formulas obtained in Section~\ref{sec:clform}, allows us
	to investigate the asymptotic behaviour of the ideal zero-product probability.
	Rather than merely proving that $\zeta_k(R)$ converges to $1$, we determine the
	precise asymptotic behaviour of the error term $1-\zeta_k(R)$, thereby obtaining
	a quantitative description of the rate of convergence.


    \begin{theorem} \label{thm:asymest}
    	Let $(R,M)$ be a finite chain ring with Loewy length $n$. Then
    	\[
    	1-\zeta_k(R)
    	\sim
    	\frac{k^{\,n-1}}
    	{(n-1)!(n+1)^k},
    	\qquad k\to\infty.
    	\]
    	Equivalently,
    	\[
    	\zeta_k(R)
    	=
    	1-
    	\frac{k^{\,n-1}}
    	{(n-1)!(n+1)^k}
    	\left(1+o(1)\right),
    	\qquad k\to\infty.
    	\]
    \end{theorem}
    
    \begin{proof}
    	By Corollary~\ref{cor:closed-form},
    	\[
    	1-\zeta_k(R)
    	=
    	\frac{1}{(n+1)^k}
    	\sum_{t=0}^{n-1}N_k(t),
    	\]
    	where
    	\[
    	N_k(t)
    	=
    	\#\{(a_1,\ldots,a_k)\in\{0,1,\ldots,n\}^k:
    	a_1+\cdots+a_k=t\}.
    	\]
    	
    	Since $t<n$, the upper bound $a_i\le n$ is never active. Hence
    	\[
    	N_k(t)
    	=
    	\binom{t+k-1}{t}.
    	\]
    	
    	For fixed $t$,
    	\[
    	\binom{t+k-1}{t}
    	=
    	\frac{k^t}{t!}
    	+O(k^{t-1}),
    	\qquad k\to\infty.
    	\]
    	
    	Therefore,
    	\[
    	N_k(n-1)
    	=
    	\frac{k^{\,n-1}}{(n-1)!}
    	+O(k^{\,n-2}),
    	\]
    	whereas, for every $0\le t\le n-2$,
    	\[
    	N_k(t)=O(k^{\,n-2}).
    	\]
    	
    	Consequently,
    	\[
    	\sum_{t=0}^{n-1}N_k(t)
    	=
    	\frac{k^{\,n-1}}{(n-1)!}
    	+O(k^{\,n-2}).
    	\]
    	
    	Substituting this into the expression for
    	$1-\zeta_k(R)$ yields
    	\[
    	1-\zeta_k(R)
    	=
    	\frac{1}{(n+1)^k}
    	\left(
    	\frac{k^{\,n-1}}{(n-1)!}
    	+O(k^{\,n-2})
    	\right),
    	\]
    	which is equivalent to
    	\[
    	1-\zeta_k(R)
    	\sim
    	\frac{k^{\,n-1}}
    	{(n-1)!(n+1)^k}.
    	\]
    	This completes the proof.
        \end{proof}
    
    	\begin{corollary}\label{thm:limit}
    		For every finite chain ring,
    		\[
    		\lim_{k\to\infty}\zeta_k(R)=1.
    		\]
    	\end{corollary}
    
    \begin{proof}
    	
    	Although this follows immediately from the preceding theorem, we present
    	below an alternative proof based on the probabilistic interpretation of
    	Theorem~\ref{thm:probabilistic}, which provides additional insight into the
    	limiting behaviour of the invariant. Let
    	
    	\[
    	S_k=X_1+\cdots+X_k
    	\]
    	
    	be as in Theorem~\ref{thm:probabilistic}. Since
    	
    	\[
    	X_i\ge0
    	\]
    	
    	for every $i$,
    	
    	\[
    	S_k
    	\]
    	
    	is nondecreasing in $k$. Moreover,
    	
    	\[
    	\Pr(X_i=n)=\frac1{n+1}>0.
    	\]
    	
    	Hence, by the second Borel--Cantelli lemma,
    	
    	\[
    	X_i=n
    	\]
    	
    	occurs infinitely often almost surely. Consequently,
    	
    	\[
    	S_k\longrightarrow\infty
    	\]
    	
    	almost surely. Since
    	
    	\[
    	\zeta_k(R)
    	=
    	\Pr(S_k\ge n),
    	\]
    	
    	it follows that
    	
    	\[
    	\Pr(S_k\ge n)\longrightarrow1,
    	\]
    	
    	and therefore
    	
    	\[
    	\lim_{k\to\infty}\zeta_k(R)=1.
    	\]
    	
    \end{proof}

	\begin{remark}
		Another alternative proof of Corollary~\ref{thm:limit} follows from the Strong Law of Large Numbers. Indeed, if
		\[
		X_1,X_2,\ldots
		\]
		are as in Theorem~\ref{thm:probabilistic}, then
		\[
		\mathbb{E}[X_i]=\frac{n}{2}>0.
		\]
		Hence
		\[
		\frac{S_k}{k}\longrightarrow \frac{n}{2}
		\qquad\text{almost surely},
		\]
		which implies that $S_k\to\infty$ almost surely. Consequently,
		\[
		\Pr(S_k\ge n)\longrightarrow 1,
		\]
		and therefore
		\[
		\lim_{k\to\infty}\zeta_k(R)=1.
		\]
	\end{remark}
	The asymptotic estimate obtained in Theorem~\ref{thm:asymest} suggests that the error term
	\[
	1-\zeta_k(R)
	\]
	possesses additional combinatorial structure. In fact, it is strictly log-concave, as is evident from its closed-form expression.
	\begin{theorem}
		Let $(R,M)$ be a finite chain ring of Loewy length $n$, and define
		\[
		a_k(R)=1-\zeta_k(R).
		\]
		Then, the sequence $\{a_k(R)\}_{k\ge2}$ is strictly log-concave; that is,
		\[
		a_k(R)^2>a_{k-1}(R)a_{k+1}(R)
		\quad \text{for} \quad k\ge2.
		\]
	\end{theorem}
	
	\begin{proof}
		By Corollary~\ref{cor:closed-form},
		\[
		a_k(R)
		=
		\frac{1}{(n+1)^k}
		\sum_{t=0}^{n-1}N_k(t).
		\]
		Since $t<n$, the upper bound $a_i\le n$ is never active, and therefore
		\[
		N_k(t)=\binom{k+t-1}{t}.
		\]
		Hence
		\[
		a_k(R)
		=
		\frac{1}{(n+1)^k}
		\sum_{t=0}^{n-1}
		\binom{k+t-1}{t}.
		\]
		Using the standard hockey-stick identity,
		\[
		\sum_{t=0}^{m}\binom{k+t-1}{t}
		=
		\binom{k+m}{m},
		\]
		with $m=n-1$, we obtain
		\[
		a_k(R)
		=
		\frac{\binom{k+n-1}{n-1}}
		{(n+1)^k}.
		\]
		
		Next,
		\[
		\frac{a_{k+1}(R)}{a_k(R)}
		=
		\frac{k+n}{(k+1)(n+1)}.
		\]
		Similarly,
		\[
		\frac{a_k(R)}{a_{k-1}(R)}
		=
		\frac{k+n-1}{k(n+1)}.
		\]
		Since
		\[
		\frac{k+n}{k+1}
		<
		\frac{k+n-1}{k},
		\]
		for every $n>1$, it follows that
		\[
		\frac{a_{k+1}(R)}{a_k(R)}
		<
		\frac{a_k(R)}{a_{k-1}(R)},
		\]
		which is equivalent to
		\[
		a_k(R)^2>
		a_{k-1}(R)a_{k+1}(R).
		\]
		Thus $\{a_k(R)\}_{k\ge2}$ is strictly log-concave.
	\end{proof}
	\begin{remark}
		
		Theorem~\ref{thm:probabilistic} identifies the ideal zero-product probability
		with the tail probability of a random walk having bounded increments.
		Consequently, a variety of probabilistic techniques, including concentration
		inequalities, large deviation estimates, local limit theorems and renewal
		methods, may be employed to investigate the asymptotic behaviour of
		$\zeta_k(R)$. Such investigations appear to be promising directions for future
		research.
		
	\end{remark}
	
	\section{Comparison with the classical zero-product probability}\label{sec:comparison}
	
	The ideal zero-product probability introduced in this paper is naturally related to the classical
	zero-product probability of finite rings, which has been studied extensively in the literature \cite{dolvzan2022probability,esmkhani2018probability,esmkhani2019characterization, Sharma2015,sarma2025probabilityproductkelements}.
	For a finite commutative ring $R$, define
	\[
	P_2(R)=
	\frac{\left|\{(a,b)\in R^2:ab=0\}\right|}{|R|^2},
	\]
	that is, the probability that two independently and uniformly chosen ring elements have product
	equal to zero.
	
	Although both $P_2(R)$ and $\zeta_2(R)$ measure the occurrence of zero products, they are
	fundamentally different in nature. The invariant $P_2(R)$ depends on the multiplicative behaviour
	of individual ring elements, whereas $\zeta_2(R)$ measures the interaction of ideals under ideal
	multiplication. Consequently, the two invariants capture different aspects of the algebraic
	structure of a finite ring. For finite chain rings, the distinction becomes particularly transparent. In this case, the classical zero-product probability admits an explicit formula depending on both the Loewy length and the residue field size, whereas the ideal zero-product probability depends only on the Loewy length.
	
	\begin{theorem}
		Let $(R,M)$ be a finite chain ring with residue field of cardinality
		\[
		q=|R/M|
		\]
		and Loewy length $n$. Let $P_2(R)$ be defined as above. Then
		\[
		P_2(R)
		=
		\frac{(n+1)q-n}{q^{\,n+1}}.
		\]
	\end{theorem}
	
	\begin{proof}
		For $0\le i\le n$, let
		\[
		N_i=
		\begin{cases}
			(q-1)q^{\,n-i-1}, & 0\le i<n,\\[1mm]
			1, & i=n.
		\end{cases}
		\]
		Since every nonzero element belongs to a unique set
		\[
		M^i\setminus M^{i+1},
		\qquad
		0\le i<n,
		\]
		and $M^n=(0)$, the quantity $N_i$ is precisely the number of elements having
		valuation $i$.
		
		Let $v:R\to\{0,1,\ldots,n\}$ denote the valuation on $R$, defined by
		\[
		v(a)=
		\begin{cases}
			i,&a\in M^i\setminus M^{i+1},\ 0\le i<n,\\
			n,&a=0.
		\end{cases}
		\]
		Since
		\[
		M^iM^j=M^{\min\{i+j,n\}},
		\]
		by Lemma~\ref{lem:ideal-product}, we obtain
		\[
		ab=0
		\quad\Longleftrightarrow\quad
		v(a)+v(b)\ge n.
		\]
		Hence
		\[
		\left|\{(a,b)\in R^2:ab=0\}\right|
		=
		\sum_{\substack{0\le i,j\le n\\ i+j\ge n}}
		N_iN_j.
		\]
		
		We evaluate this sum by separating the cases in which one of the elements is
		zero.
		
		First,
		\[
		\sum_{\substack{0\le i,j<n\\ i+j\ge n}}
		N_iN_j
		=
		(q-1)^2
		\sum_{i=1}^{n-1}
		\sum_{j=n-i}^{n-1}
		q^{\,2n-i-j-2}.
		\]
		For a fixed $i$, writing $j=n-i+m$ with
		$0\le m\le i-1$, we obtain
		\[
		\sum_{j=n-i}^{n-1}
		q^{\,2n-i-j-2}
		=
		\sum_{m=0}^{i-1}
		q^{\,n-m-2}
		=
		\frac{q^{\,n-1}-q^{\,n-i-1}}{q-1}.
		\]
		Therefore,
		\begin{align*}
			\sum_{\substack{0\le i,j<n\\ i+j\ge n}}
			N_iN_j
			&=
			(q-1)
			\sum_{i=1}^{n-1}
			\left(q^{\,n-1}-q^{\,n-i-1}\right)\\
			&=
			(q-1)(n-1)q^{\,n-1}
			-
			\left(q^{\,n-1}-1\right)\\
			&=
			(nq-n-q)q^{\,n-1}+1.
		\end{align*}
		
		Next,
		\[
		\sum_{i=0}^{n-1}N_i
		=
		q^n-1,
		\]
		since these are precisely the nonzero elements of $R$.
		
		Hence
		\begin{align*}
			\left|\{(a,b)\in R^2:ab=0\}\right|
			&=
			\bigl((nq-n-q)q^{\,n-1}+1\bigr)
			+2(q^n-1)+1\\
			&=
			\bigl((n+1)q-n\bigr)q^{\,n-1}.
		\end{align*}
		
		Finally,
		\[
		|R|=q^n,
		\]
		so dividing by $|R|^2=q^{2n}$ gives
		\[
		P_2(R)
		=
		\frac{\bigl((n+1)q-n\bigr)q^{\,n-1}}{q^{2n}}
		=
		\frac{(n+1)q-n}{q^{\,n+1}},
		\]
		as required.
	\end{proof}
	
	\begin{corollary} \label{cor:difference}
		Let $(R,M)$ be a finite chain ring with residue field of cardinality $q$ and
		Loewy length $n$. Then
		\[
		\zeta_2(R)=\frac{n+2}{2(n+1)},
		\qquad
		P_2(R)=\frac{(n+1)q-n}{q^{\,n+1}}.
		\]
		Consequently, the ideal zero-product probability depends only on the Loewy
		length, whereas the classical zero-product probability depends on both the
		Loewy length and the cardinality of the residue field.
	\end{corollary}
	
	\begin{proof}
		The formula for $\zeta_2(R)$ is Corollary~\ref{cor:valueofzeta2}, while the formula for $P_2(R)$
		is given by the preceding theorem.
	\end{proof}
	
	This comparison demonstrates that the ideal zero-product probability is genuinely different from
	the classical element-based probability. Whereas $P_2(R)$ is sensitive to the distribution of
	zero-divisors among the elements of the ring, $\zeta_2(R)$ reflects the multiplicative structure
	of the ideal lattice. Thus the two invariants provide complementary information about finite
	commutative rings.
\begin{remark}	
For fixed Loewy length $n$,
\[
\lim_{q\to\infty}P_2(R)=0,
\]
whereas
\[
\zeta_2(R)=\frac{n+2}{2(n+1)}
\]
is independent of $q$. Thus the classical zero-product probability becomes arbitrarily small as the residue field grows, while the ideal zero-product probability remains unchanged. This further illustrates that the two invariants capture fundamentally different aspects of the structure of finite chain rings.
\end{remark}
The explicit formulas obtained above permit a direct comparison between the
classical zero-product probability and the ideal zero-product probability.
The following result shows that, for finite chain rings, the ideal
zero-product probability always dominates the classical zero-product
probability, and equality occurs only in the unique smallest finite field.
This highlights the fact that the two invariants, although related, measure
fundamentally different aspects of the multiplicative structure of a finite
ring.	

\begin{theorem}\label{thm:ReqF2}
	Let $(R,M)$ be a finite chain ring with residue field of cardinality $q$ and
	Loewy length $n$. Then
	\[
	P_2(R)\le \zeta_2(R),
	\]
	with equality if and only if
	\[
	R\cong \mathbb{F}_2.
	\]
\end{theorem}

\begin{proof}
	By Corollary~\ref{cor:difference},
	\[
	P_2(R)=\frac{(n+1)q-n}{q^{\,n+1}}
	\quad\text{and}\quad
	\zeta_2(R)=\frac{n+2}{2(n+1)}.
	\]
	Thus
	\[
	P_2(R)\le \zeta_2(R)
	\]
	is equivalent to
	\[
	2(n+1)\bigl((n+1)q-n\bigr)
	\le
	(n+2)q^{\,n+1}.
	\]
	Define
	\[
	F(n,q)
	=
	(n+2)q^{\,n+1}
	-
	2(n+1)^2q
	+
	2n(n+1).
	\]
	Then the desired inequality is precisely
	\[
	F(n,q)\ge0.
	\]
	For fixed $n$,
	\[
	\frac{\partial F}{\partial q}
	=
	(n+1)\bigl((n+2)q^n-2(n+1)\bigr).
	\]
	Since $q\ge2$,
	\[
	q^n\ge2^n,
	\]
	and
	\[
	(n+2)2^n>2(n+1)
	\]
	for every integer $n\ge1$, it follows that
	\[
	\frac{\partial F}{\partial q}>0.
	\]
	Hence $F(n,q)$ is strictly increasing as a function of $q$. Therefore it
	suffices to verify the inequality for $q=2$. Now
	\begin{align*}
		F(n,2)
		&=(n+2)2^{\,n+1}-4(n+1)^2+2n(n+1)\\
		&=(n+2)\bigl(2^{\,n+1}-2(n+1)\bigr).
	\end{align*}
	Since
	\[
	2^n\ge n+1
	\]
	for every integer $n\ge1$, we obtain
	\[
	F(n,2)\ge0.
	\]
	Consequently,
	\[
	F(n,q)\ge F(n,2)\ge0,
	\]
	and therefore
	\[
	P_2(R)\le\zeta_2(R).
	\]
	Equality holds if and only if
	\[
	F(n,q)=0.
	\]
	Since $F(n,q)$ is strictly increasing in $q$, equality is possible only for
	$q=2$. Moreover,
	\[
	F(n,2)=0
	\]
	is equivalent to
	\[
	2^n=n+1,
	\]
	whose unique positive integer solution is $n=1$. A finite chain ring of
	Loewy length one is a field, and since $q=2$, it follows that
	\[
	R\cong\mathbb{F}_2.
	\]
	Conversely,
	\[
	P_2(\mathbb{F}_2)
	=\frac34
	=\zeta_2(\mathbb{F}_2),
	\]
	so equality indeed holds.
\end{proof}
The comparison established above naturally extends beyond finite chain rings.
Indeed, every finite principal ideal ring is a finite direct product of finite
chain rings, so it is enough to establish the multiplicativity of the classical
zero-product probability under direct products.
\begin{theorem}\label{thm:prod2ofP2}
	Let $R_1$ and $R_2$ be two finite commutative rings. Then
	\[
	P_2(R_1\times R_2)=P_2(R_1)P_2(R_2).
	\]
	Consequently, if
	\[
	R\cong\prod_{i=1}^{m}R_i,
	\]
	where each $R_i$ is a finite commutative ring, then
	\[
	P_2(R)=\prod_{i=1}^{m}P_2(R_i).
	\]
\end{theorem}

\begin{proof}
	Since multiplication in the direct product is componentwise,
	\[
	(a_1,a_2)(b_1,b_2)=0
	\]
	if and only if
	\[
	a_1b_1=0
	\quad\text{and}\quad
	a_2b_2=0.
	\]
	Hence
	\[
	\left|\{((a_1,a_2),(b_1,b_2)):(a_1,a_2)(b_1,b_2)=0\}\right|
	=
	\left|\{(a_1,b_1):a_1b_1=0\}\right|
	\left|\{(a_2,b_2):a_2b_2=0\}\right|.
	\]
	Since \(
	|R_1\times R_2|=|R_1||R_2|,
	\) dividing by $|R_1\times R_2|^2$ gives
	\[
	P_2(R_1\times R_2)=P_2(R_1)P_2(R_2).
	\]
	The general case follows immediately by induction on the number of direct factors.
\end{proof}
The multiplicativity of both invariants, together with the canonical
decomposition of finite principal ideal rings into finite chain rings, enables
us to obtain the following global comparison theorem. In particular, it
identifies finite Boolean rings as the unique finite principal ideal rings for
which the classical and ideal zero-product probabilities coincide.
\begin{theorem}\label{thm:p2lessz2}
	Let $R$ be a finite principal ideal ring. Then
	\[
	P_2(R)\le \zeta_2(R),
	\]
	with equality if and only if $R$ is a finite Boolean ring.
\end{theorem}

\begin{proof}
	Since $R$ is a finite principal ideal ring, there exist finite chain rings
	$R_1,R_2,\ldots,R_m$ such that
	\[
	R\cong\prod_{i=1}^{m}R_i.
	\]
	By Theorem~\ref{thm:prod2ofP2},
	\[
	P_2(R)=\prod_{i=1}^{m}P_2(R_i),
	\]
	while Corollary~\ref{cor:finite-product} gives
	\[
	\zeta_2(R)=\prod_{i=1}^{m}\zeta_2(R_i).
	\]
	Since each $R_i$ is a finite chain ring, Theorem~\ref{thm:ReqF2} yields
	\[
	P_2(R_i)\le \zeta_2(R_i),
	\qquad
	1\le i\le m.
	\]
	Therefore,
	\[
	P_2(R)
	=
	\prod_{i=1}^{m}P_2(R_i)
	\le
	\prod_{i=1}^{m}\zeta_2(R_i)
	=
	\zeta_2(R),
	\]
	proving the desired inequality. Suppose now that
	\[
	P_2(R)=\zeta_2(R).
	\]
	Then
	\[
	\prod_{i=1}^{m}P_2(R_i)
	=
	\prod_{i=1}^{m}\zeta_2(R_i).
	\]
	Since
	\[
	0<P_2(R_i)\le \zeta_2(R_i)
	\]
	for every $i$, equality of the products is possible only if
	\[
	P_2(R_i)=\zeta_2(R_i)
	\]
	for every $i$. By Theorem~\ref{thm:ReqF2}, this occurs if and only if
	\[
	R_i\cong\mathbb{F}_2
	\]
	for every $i$. Consequently,
	\[
	R\cong
	\prod_{i=1}^{m}\mathbb{F}_2,
	\]
	which is precisely a finite Boolean ring. 
	
	Conversely, if $R$ is a finite Boolean ring, then
	\[
	R\cong\prod_{i=1}^{m}\mathbb{F}_2
	\]
	for some positive integer $m$. Since
	\[
	P_2(\mathbb{F}_2)=\zeta_2(\mathbb{F}_2)=\frac34,
	\]
	multiplicativity gives
	\[
	P_2(R)
	=
	\prod_{i=1}^{m}P_2(\mathbb{F}_2)
	=
	\prod_{i=1}^{m}\zeta_2(\mathbb{F}_2)
	=
	\zeta_2(R).
	\]
	This completes the proof.
\end{proof}

	\section{Open problems}\label{sec:OP}
	
	The ideal zero-product probability introduced in this paper gives rise to
	several natural questions whose investigation may lead to a deeper
	understanding of the interaction between ideal theory, finite ring theory, and
	enumerative combinatorics.
	
	\bigskip
	
	\noindent\textbf{Problem 1 (Universality beyond chain rings).}
	Corollary~\ref{cor:loewy} shows that, for finite chain rings, the ideal
	zero-product probability is completely determined by the Loewy length.
	It is natural to ask whether this universality phenomenon extends to
	broader classes of finite rings. In particular, characterize all finite
	commutative rings $R$ for which $\zeta_k(R)$ depends only on the Loewy
	length. Does such a characterization hold for finite principal ideal
	rings, finite local rings, finite valuation rings, or, more generally,
	finite Artinian rings?
	
	\bigskip
	
	\noindent\textbf{Problem 2.}
		Investigate the dependence of $\zeta_k(R)$ on classical ring-theoretic invariants by establishing sharp upper and lower bounds, and extend the theory to finite noncommutative rings, where separate notions based on left, right, and two-sided ideals may be introduced and compared.

		\noindent\textbf{Problem 3.}
		Investigate the hypergraph whose vertices are the ideals of a finite ring and
		whose hyperedges are the ordered $k$-tuples
		\[
		(I_1,\ldots,I_k)
		\]
		satisfying
		\[
		I_1I_2\cdots I_k=(0).
		\]
		Study the relationship between hypergraph invariants and the ideal
		zero-product probability.

		\noindent\textbf{Problem 4.}
		Investigate the extent to which $\zeta_k(R)$ is determined by the algebraic structure of the ring. In particular, determine whether isomorphic ideal lattices imply equal ideal zero-product probabilities, and derive explicit formulas for broader classes of finite rings, including finite valuation rings, Galois rings, and finite principal ideal local rings.

		\noindent\textbf{Problem 5.}
		Study the asymptotic and enumerative aspects of the sequence $\{\zeta_k(R)\}_{k\ge2}$, including its rate of convergence, possible limit laws, generating functions, recurrence relations, and other combinatorial properties.

	\bibliographystyle{abbrv}
	\bibliography{prodideal}

\begin{thebibliography}{10}

\bibitem{buckley2014finite}
S.~M. Buckley, D.~MacHale, and {\'A}.~N. Sh{\'e}.
\newblock Finite rings with many commuting pairs of elements.
\newblock {\em preprint}, 2014.

\bibitem{burness2023commuting}
T.~C. Burness, R.~Guralnick, A.~Moret{\'o}, and G.~Navarro.
\newblock On the commuting probability of p-elements in a finite group.
\newblock {\em Algebra \& Number Theory}, 17(6):1209--1229, 2023.

\bibitem{camillo1974loewy}
V.~Camillo and K.~Fuller.
\newblock On loewy length of rings.
\newblock {\em Pacific Journal of Mathematics}, 53(2):347--354, 1974.

\bibitem{dolvzan2022probability}
D.~Dol{\v{z}}an.
\newblock The probability of zero multiplication in finite rings.
\newblock {\em Bulletin of the Australian Mathematical Society}, 106(1):83--88,
  2022.

\bibitem{doostie2007finite}
H.~Doostie and L.~Pourfaraj.
\newblock Finite rings and loop rings involving the commuting regular elements.
\newblock In {\em International Mathematical Forum}, volume~2, pages
  2579--2586, 2007.

\bibitem{DuttaNath2015}
J.~Dutta, D.~K. Basnet, and R.~K. Nath.
\newblock On commuting probability of finite rings.
\newblock {\em Indagationes Mathematicae}, 28(2):372--382, 2017.

\bibitem{esmkhani2018probability}
M.~Esmkhani and S.~Jafarian~Amiri.
\newblock The probability that the multiplication of two ring elements is zero.
\newblock {\em Journal of Algebra and Its Applications}, 17(03):1850054, 2018.

\bibitem{esmkhani2019characterization}
M.~Esmkhani and S.~Jafarian~Amiri.
\newblock Characterization of rings with nullity degree at least 1/4.
\newblock {\em Journal of Algebra and Its Applications}, 18(04):1950076, 2019.

\bibitem{Gustafson1973}
W.~H. Gustafson.
\newblock What is the probability that two group elements commute?
\newblock {\em The American mathematical monthly}, 80(9):1031--1034, 1973.

\bibitem{jaklivc2010closed}
G.~Jakli{\v{c}}, V.~Vitrih, and E.~{\v{Z}}AGAR.
\newblock Closed form formula for the number of restricted compositions.
\newblock {\em Bulletin of the Australian Mathematical Society},
  81(2):289--297, 2010.

\bibitem{MacHale1976}
D.~MacHale.
\newblock Commutativity in finite rings.
\newblock {\em The American Mathematical Monthly}, 83(1):30--32, 1976.

\bibitem{McDonald1974}
B.~R. McDonald.
\newblock {\em Finite Rings with Identity}, volume~28 of {\em Pure and Applied
  Mathematics}.
\newblock Marcel Dekker, New York, 1974.

\bibitem{mohammed2022probability}
H.~M. Mohammed~Salih.
\newblock On the probability of zero divisor elements in group rings.
\newblock {\em International Journal of Group Theory}, 11(4):253--257, 2022.

\bibitem{sarma2025probabilityproductkelements}
D.~Sarma and T.~Subedi.
\newblock The probability that the product of k elements in a finite ring is
  zero, 2025.

\bibitem{Sharma2015}
D.~Sarma and T.~Subedi.
\newblock The probability that the product of three elements in a finite ring
  is zero.
\newblock {\em Journal of Algebra and Its Applications}, 25(3):2550349, 2026.

\bibitem{StanleyEC1}
R.~P. Stanley.
\newblock {\em Enumerative combinatorics volume 1 second edition}.
\newblock 2011.

\bibitem{Wilf}
H.~S. Wilf.
\newblock {\em Generating Functionology}.
\newblock CRC press, 2005.

\end{thebibliography}
\end{document}